\documentclass[11pt]{amsart}

\usepackage{amssymb}
\usepackage{mathrsfs}
\usepackage{amsfonts}
\usepackage{latexsym,amsmath,amsthm,amssymb,amsfonts}
\usepackage[usenames]{color}
\usepackage{xspace,colortbl}
\usepackage{graphicx}
\usepackage{tipa}

\newcommand{\be}{\begin{equation}}
\newcommand{\ee}{\end{equation}}
\newcommand{\beq}{\begin{eqnarray}}
\newcommand{\eeq}{\end{eqnarray}}

\newtheorem{prop}{Proposition}[section]

\newtheorem{remark}[prop]{Remark}

\def\begeq{\begin{equation}}
\def\endeq{\end{equation}}

\def\odot{\setbox0=\hbox{$\bigcirc$}\relax \mathbin {\hbox
to0pt{\raise.5pt\hbox to\wd0{\hfil $\wedge$\hfil}\hss}\box0 }}

\numberwithin{equation} {section}

\numberwithin{equation}{section}
\newtheorem{theorem}{\bf Theorem}[section]

\newtheorem{lemma}[theorem]{\bf Lemma}

\allowdisplaybreaks

\begin{document}

\title[inverse mean curvature flow inside a cone in warped products] {inverse mean curvature flow inside a cone in warped products}

\author{
Li Chen, Jing Mao$^{\ast}$, Ni Xiang and Chi Xu}
\address{
Faculty of Mathematics and Statistics, Key Laboratory of Applied
Mathematics of Hubei Province, Hubei University, Wuhan 430062,
China. } \email{chernli@163.com, jiner120@163.com,
nixiang@hubu.edu.cn}

\thanks{
$\ast$ Corresponding author}

\date{}
\begin{abstract}
Given a convex cone in the \emph{prescribed} warped product, we
consider hypersurfaces with boundary which are star-shaped with
respect to the center of the cone and which meet the cone
perpendicularly. If those hypersurfaces inside the cone evolve along
the inverse mean curvature flow, then, by using the convexity of the
cone, we can prove that this evolution exists for all the time and
the evolving hypersurfaces converge smoothly to a piece of round
sphere as time tends to infinity.
\end{abstract}

\maketitle {\it \small{{\bf Keywords}: Inverse mean curvature flow,
cone, warped products.}

 }

\section{Introduction}

Unlike the mean curvature flow, which is a shrinking flow, the
inverse mean curvature flow (IMCF for short), which says a
submanifold evolves along its outward normal direction with a speed
equal to the reciprocal of the mean curvature, in general is an
expanding flow.

A classical result for IMCF is due to Gerhardt \cite{cg}, who proved
that if a closed, smooth, star-shaped hypersurface with strictly
positive mean curvature evolves along the IMCF, then the flow exists
for all the time and, after rescaling, the evolving hypersurfaces
converge to a round sphere as time tends to infinity (see also
\cite{ju}). The star-shaped assumption is important in the study of
IMCF, since it allows us to equivalently transform the evolution
equation of IMCF, which, in local coordinates of the initial
submanifold, corresponds to a system of second-order parabolic
partial differential equations (PDEs for short) with some specified
initial conditions, into a scalar second-order parabolic PDE which
seems to be relatively easy to deal with.

In general, singularities may occur in finite time for
non-starshaped submanifolds evolving under the IMCF. By defining a
notion of weak solutions to IMCF, Huisken and Ilmanen \cite{hi1,hi2}
proved the Riemannian Penrose inequality by using the IMCF approach.
One of the reasons why people pay attention to the study of IMCF is
that one can use IMCF to derive some interesting geometric
inequalities like what Huisken and Ilmanen have done. In fact, there
already exist some results like this. For instance, we know that for
a closed convex surface $\mathcal {M}$ in the $3$-dimensional
Euclidean space $\mathbb{R}^3$, the classical Minkowski inequality
is
\begin{eqnarray*}
\int\limits_{\mathcal{M}}Hd\mu\geqslant\sqrt{16\pi|\mathcal{M}|},
\end{eqnarray*}
where $H$ is the mean curvature, $|\mathcal{M}|$ is the area of
$\mathcal{M}$, and $d\mu$ is the volume density of $\mathcal{M}$.
This result can be generalized to the high dimensional case. In
fact, for a convex hypersurface $\mathcal {M}$ in $\mathbb{R}^n$,
one has
\begin{eqnarray*}
\int\limits_{\mathcal{M}}Hd\mu\geqslant(n-1)|\mathbb{S}^{n-1}|^{\frac{1}{n-1}}|\mathcal{M}|^{\frac{n-2}{n-1}},
\end{eqnarray*}
where $|\mathbb{S}^{n-1}|$ denotes the area of the unit sphere in
$\mathbb{S}^{n-1}$ in $\mathbb{R}^n$. By using the method of IMCF,
the above Minkowski inequality has been proven to be valid for mean
convex and star-shaped hypersurfaces in $\mathbb{R}^n$ also (cf.
\cite{gl,gmtz}). Also using the method of IMCF, Brendle, Hung and
Wang \cite{bhw} proved a sharp Minkowski inequality for mean convex
and star-shaped hypersurfaces in the $n$-dimensional ($n\geqslant3$)
anti-de Sitter-Schwarzschild manifold, which generalized the related
conclusions in the Euclidean space mentioned above.

The first and second authors have been working on IMCF for several
years and have also obtained some interesting results. For instance,
Chen and Mao \cite{cm} considered the evolution of a smooth,
star-shaped and $F$-admissible ($F$ is a 1-homogeneous function of
principle curvatures satisfying some suitable conditions) embedded
closed hypersurface in the $n$-dimensional ($n\geqslant3$) anti-de
Sitter-Schwarzschild manifold along its outward normal direction has
a speed equal to $1/F$ (clearly, this evolution process is a natural
generalization of IMCF, and we call it \emph{inverse curvature
flow}. We write as ICF for short), and they proved that this ICF
exists for all the time and, after rescaling, the evolving
hypersurfaces converge to a sphere as time tends to infinity. This
interesting conclusion has been improved by Chen, Mao and Zhou
\cite{cmz} to the situation that the ambient space is a warped
product $I\times_{\lambda(r)}N^{n}$ with $I$ an unbounded interval
of $\mathbb{R}$ (i.e., the set of real numbers) and $N^{n}$ a
Riemannian manifold of nonnegative Ricci curvature.

Suppose $N$ and $B$ are semi-Riemannian manifolds with metrics
$g_{N}$ and $g_{B}$, and let $f>0$ be a smooth function on $N$. The
\emph{warped product} $N\times_{f}B$ is the product manifold
furnished with the metric tensor
$g=\pi^{\ast}(g_{N})+(f\circ\pi)^{2}\sigma^{\ast}(g_{B})$, where
$\pi$ and $\sigma$ are the projections of $N\times{B}$ onto $N$ and
$B$, respectively. $N$ and $B$ are called the \emph{base} and the
\emph{fiber} of  $N\times_{f}B$, respectively.  Clearly,
$I\times_{\lambda(r)}N^{n}$ mentioned above is a special warped
product with base $N^{n}$ and fiber $I\subset\mathbb{R}$. Comparing
with general Riemannian manifolds, warped products have some
interesting and useful properties (see, for instance, \cite[Appendix
A]{m2} or \cite[Appendix A]{mdw}). We call special warped products
$[0,\ell)\times_{f(x)}\mathbb{S}^{n-1}$, $f(0)=0$, $f'(0)=1$,
$f|_{(0,\ell)}>0$, $0<\ell\leqslant\infty$ are $n$-dimensional
($n\geqslant2$) \emph{spherically symmetric manifolds} with the base
point $p:=\{0\}\times_{f(0)}\mathbb{S}^{n-1}$ (also known as
\emph{generalized space forms}). Especially, if $\ell=\infty$,
$f(x)=x$, then
$[0,\ell)\times_{f(x)}\mathbb{S}^{n-1}\equiv\mathbb{R}^n$; if
$\ell=\infty$, $f(x)=\frac{\sinh(\sqrt{-k}x)}{\sqrt{-k}}$, then
$[0,\ell)\times_{f(x)}\mathbb{S}^{n-1}\equiv\mathbb{H}^n(-k)$, i.e.,
the $n$-dimensional hyperbolic space with constant sectional
curvature $-k<0$; if $\ell=\frac{\pi}{\sqrt{k}}$,
$f(x)=\frac{\sin(\sqrt{k}x)}{\sqrt{k}}$, then after endowing a
one-point compactification topology, the closure of
$[0,\ell)\times_{f(x)}\mathbb{S}^{n-1}$ equals $\mathbb{S}^{n}(k)$,
i.e., the $n$-dimensional sphere with constant sectional curvature
$k>0$. Spherically symmetric manifolds are very nice mode spaces
which can be used to successfully improve some classical results in
Riemannian geometry (for example, Cheng's eigenvalue comparison
theorem, Bishop's volume comparison theorem, and so on). For more
details on this topic, we refer readers to, for instance,
\cite{fmi,m2,m1,m3,mdw}.

Marquardt \cite{Ma2} successfully proved that if an $n$-dimensional
($n\geqslant2$) compact $C^{2,\alpha}$-hypersurface with boundary,
which meets a given cone in $\mathbb{R}^{n+1}$ perpendicularly and
is star-shaped with respect to the center of the cone, evolves along
the IMCF, then the flow exists for all the time and, after
rescaling, the evolving hypersurfaces converge to a piece of the
round sphere as time tends to infinity. Based on our experience in
\cite{cm,cmz}, we would like to know ``\emph{if we replace the
ambient space $\mathbb{R}^{n+1}$ in \cite{Ma2} by a warped product
 $I\times_{\lambda(r)}N^{n}$ with $I\subseteq\mathbb{R}$ an unbounded interval, whether the IMCF exists
for all the time or not? What about the convergence if we have the
long-time existence?}". The purpose of this paper is trying to
answer this question.

Assume that, as before, $I$ is an unbounded interval of $\mathbb{R}$
and $N^{n}$ is a Riemannian manifold with metric $g_{N}$. Naturally,
$I\times_{\lambda(r)}N^{n}$ is a warped product with the warping
function $\lambda(r)$ defined on $I$ and the metric given as follows
\begin{eqnarray*}
g=dr\otimes dr +\lambda^{2}(r)g_{N}.
\end{eqnarray*}
  Let $M^{n}\subset
N^{n}$ be a portion of $N^{n}$ such that $\Sigma^{n}:=\{(y(x),x)\in
I\times_{\lambda(r)}N^{n}|y(x)>0,x\in\partial M^{n}\}$ be the
boundary of a smooth convex cone. We can prove the following
conclusion.

\begin{theorem}\label{main1.1}
Let $I\times_{\lambda(r)}N^{n}$ be an $(n+1)$-dimensional ($n
\geqslant 2$) warped product with the warping function $\lambda(r)$
satisfying $\lambda(r)>0$, $0<\lambda'(r)\leqslant C$ and
$0\leqslant\lambda^{1+\alpha}(r)\lambda''(r)\leqslant C$ for some
positive constants $\alpha$, $C$ on $I^{\circ}$ (if $I$ has an
endpoint $a$, then $I^{\circ}=I\backslash\{a\}$; if $I$ does not
have endpoint, then $I^{\circ}=I$), where $I$ denotes an unbounded
interval of $\mathbb{R}$ and $N^n$ is an $n$-dimensional Riemannian
manifold with nonnegative Ricci curvature. Let $\Sigma^n \subset
I\times_{\lambda(r)}N^{n}$ be the boundary of a smooth, convex cone
that is centered at some interior point of $M^n$ and has the outward
unit normal vector $\mu$. Let $F_0:M^n \rightarrow
I\times_{\lambda(r)}N^{n}$ such that $M^n_0:= F_0(M^n)$ is a compact
$C^{2,\alpha}$-hypersurface which is star-shaped with respect to the
center of the cone and has a strictly positive principal curvature.
Assume furthermore that $M^n_0$ meets $\Sigma^n$ orthogonally.That
is
\begin{eqnarray*}
F_0(\partial {M^n}) \subset {\Sigma}^n,\quad \qquad{\langle\mu \circ
F_0,\vec{\nu}_{0} \circ F_0\rangle|}_{\partial {M^n}}=0,
\end{eqnarray*}
where $\vec{\nu}_{0}$ is the outward unit normal to $M^n_0$. Then
there exists a unique embedding
\begin{eqnarray*}
F \in C^{2+\alpha,\frac{1+\alpha}{2}}\left(M^n\times
[0,\infty),I\times_{\lambda(r)}N^{n}\right) \cap
C^{\infty}\left(M^n\times
(0,\infty),I\times_{\lambda(r)}N^{n}\right)
\end{eqnarray*}
with $F(\partial {M^n},t) \subset \Sigma^n$ for $t\geqslant0$,
satisfying the following system
\begin{eqnarray*}
(\sharp) \left\{
\begin{array}{lll}
\frac{\partial F}{\partial t}=\frac{\nu}{H}\circ F \qquad &in~
M^{n}\times(0,\infty)\\
\langle\mu\circ F,\vec{\nu}\circ F\rangle=0 \qquad &on~\partial M^{n}\times(0,\infty)\\
F(\cdot,0)=F_{0} \qquad  &on~M^{n}
\end{array}
\right.
\end{eqnarray*}
where $\vec{\nu}$ is the unit normal vector to $M^n_t:=F(M^n,t)$
pointing away from the center of the cone and $H$ is the scalar mean
curvature of $M^n_t$. Moreover, after area-preserving rescaling, the
rescaled solution $\widetilde{F}(\cdot,t)$ converges smoothly to an
embedding $F_{\infty}$, mapping $M^n$ into a piece of a geodesic
sphere.
\end{theorem}

\begin{remark}
\rm{ Clearly, if $\lambda(r)=r$, $N^n=\mathbb{S}^n$ (i.e., the
$n$-dimensional Euclidean unit sphere), $I=[0,\infty)$, then
$I^{\circ}=(0,\infty)$ and
$I\times_{\lambda(r)}N^{n}\equiv\mathbb{R}^{n+1}$, Theorem
\ref{main1.1} here degenerates into \cite[Theorem 1]{Ma1}. That is,
our Theorem \ref{main1.1}
 covers \cite[Theorem 1]{Ma1} as a special case.
}
\end{remark}

This paper is organized as follows. The geometry of star-shaped
hypersurfaces in the warped product $I\times_{\lambda(r)}N^{n}$ will
be discussed in Section 2, and we will use the fact that star-shaped
hypersurfaces $M^{n}_{t}$ can be written as graphs over
$M^{n}\subset N^{n}$ to transform the first evolution equation of
the system $(\sharp)$ into a scalar second-order parabolic PDE,
which leads to the short-time existence of the IMCF. $C^0$ and
gradient estimates will be derived in Sections 3 and 4,
respectively. Higher regularity and convergence of the solution of
$(\sharp)$ will be shown in the last section.

\section{Preliminary Facts}

In this section, we would like to give some basic facts first such
that our conclusions can be explained clearly and understood well.

We want to describe the hypersurface $M^{n}_{t}$ at time $t$ as a
graph over $M^{n}\subset N^n\subset I\times_{\lambda(r)}N^{n}$, and
then we can make ansatz
\begin{eqnarray*}
\widetilde{F}:M^n \times [0,T) \rightarrow
I\times_{\lambda(r)}N^{n}:(x,t) \rightarrow \left(u(x,t),x\right)
\end{eqnarray*}
for some function $u: M^n \times [0,T) \rightarrow
I\subseteq\mathbb{R}$. Since the initial $C^{2,\alpha}$-hypersurface
is star-shaped, there exists a scalar function $u_0\in C^{2,\alpha}
(M^n_{0})$ such that the $F_0: M^n \rightarrow
I\times_{\lambda(r)}N^{n}$ has the form $x \mapsto (u_0(x),x)$. Set
$\widetilde{M}^{n}_{t}:=\widetilde{F}(M^n,t)$. Define
$p:=\widetilde{F}(x,t)$ and assume that a point on $M^n$ is
described by local coordinates $\xi^{1},\ldots,\xi^{n}$, that is,
$x=x(\xi^{1},\ldots,\xi^{n})$. Let $\partial_i$ be the corresponding
coordinate vector fields on $M^{n}\subset N^n$ and
$\sigma_{ij}=g_{N}(\partial_i,\partial_j)$ be the metric on
$M^{n}\subset N^n$. Let $u_{i}=Du_{i}$, $u_{ij}=D_{j}D_{i}u$, and
$u_{ijk}=D_{k}D_{j}D_{i}u$ denote the covariant derivatives of $u$
with respect to the metric $g_{N}$ and let $\nabla$ be the
Levi-Civita connection of $\widetilde{M}^{n}_{t}$ with respect to
the metric $\widetilde{g}$ induced from the metric $g$ of the warped
product $I\times_{\lambda(r)}N^{n}$. The tangent vector on
$\widetilde{M}^{n}_{t}$ is
\begin{eqnarray*}
\vec{e}_{i}=\partial_{i}+D_{i}u\partial_{r}
\end{eqnarray*}
and the corresponding outward unit normal vector is given by
\begin{eqnarray*}
\vec{\nu}=\frac{1}{v}\left(\partial_r-\frac{1}{\lambda^2}\nabla^j
u\partial_j\right),
\end{eqnarray*}
where $\nabla^{j}u=\sigma^{ij}\nabla_{i}u$, and
$v:=\sqrt{1+\lambda^{-2}|\nabla u|^2}$ with $\nabla u$ the gradient
of $u$. Clearly, we know that the induced metric $\widetilde{g}$ on
$\widetilde{M}^n_t$ has the form
\begin{equation}\label{2.1}
g_{ij}=\lambda^2\sigma_{ij}+\nabla _{i}u\cdot\nabla _{j}u
\end{equation}
and its inverse is given by
\begin{equation}\label{2.2}
g^{ij}=\frac{1}{\lambda^2}\left(\sigma^{ij}-\frac{\nabla^i u\nabla^j
u}{1+|\nabla
u|^2}\right)=\frac{1}{\lambda^2}\left(\sigma^{ij}-\frac{\nabla^i u
\nabla^j u}{v^{2}}\right).
\end{equation}
Let $h_{ij}dx_i \otimes dx_j$ be the second fundamental form of
$\widetilde{M}^n_t$, and then we have
\begin{eqnarray*}
h_{ij}=\langle\nabla_{\vec{e}_i}
\vec{e}_j,\vec{\nu}\rangle=-\frac{1}{v}\left(u_{i,j}-\lambda
\lambda^{\prime}\sigma_{ij}-2\frac{\lambda^{\prime}}{\lambda}u_iu_j\right),
\end{eqnarray*}
 where $u_{i,j}$ is the covariant derivative of $u$.  Define a new function
$\varphi(x,t)=\int_{c}^{u(x,t)}\frac{1}{\lambda(s)}ds$, where
$x=x(\xi^{1},\ldots,\xi^{n})$, and then the second fundamental form
can be rewritten as
\begin{eqnarray*}
h_{ij}=\frac{\lambda}{v}\left(\lambda^{\prime}(\sigma_{ij}+\varphi_i
\varphi_j)-\varphi_{i,j}\right)
\end{eqnarray*}
and
\begin{eqnarray*}
h^i_j=g^{ik}h_{jk}=\frac{\lambda^{\prime}}{\lambda
v}\delta^i_j-\frac{1}{\lambda v}\widetilde{\sigma}^{ik}\varphi_{k,j}
\qquad\qquad \mathrm{with}~~
\widetilde{\sigma}^{ij}=\sigma^{ij}-\frac{\varphi_i \varphi_j}{v^2}.
 \end{eqnarray*}
Naturally, the scalar mean curvature is given by
\begin{eqnarray*}
H=\sum_{i=1}^{n}h^i_i=\frac{n\lambda^{\prime}}{\lambda
v}-\frac{1}{\lambda
v}\sum_{i=1}^{n}\left(\sum_{k=1}^{n}\widetilde{\sigma}^{ik}\varphi_{k,i}\right).
\end{eqnarray*}
Based on the above facts and \cite{Ma2}, we can get the following
existence and uniqueness for the IMCF $(\sharp)$.
\begin{lemma} \label{lemma2.1}
Let $F_0$, $I$, $\lambda$ be as in Theorem \ref{main1.1}. Then there
exist some $T>0$, a unique solution  $u \in
C^{2+\alpha,\frac{1+\alpha}{2}}(M^n\times [0,\infty),I) \cap
C^{\infty}(M^n \times (0,\infty), I)$, where
$\varphi(x,t)=\int_{c}^{u(x,t)}\frac{1}{\lambda(s)}ds$,  of the
following system
\begin{eqnarray*}
(\widetilde{\sharp})\left\{
\begin{array}{lll}
\frac{\partial \varphi}{\partial t}=\frac{v}{\lambda
H}=\frac{v^2}{n\lambda^{\prime}-\widetilde{\sigma}^{ij}\varphi_{i,j}}
\qquad \qquad & in ~M^{n}\times(0,T)\\
\nabla_\mu \varphi =0 \qquad\qquad & on ~\partial M^{n}\times(0,T)\\
\varphi(\cdot,0)=\varphi_{0}:=\int_{c}^{u(x,0)}\frac{1}{\lambda(s)}ds
\qquad\qquad & on ~ M^{n},
\end{array}
\right.
\end{eqnarray*}
and a unique map $\psi: M^{n}\times[0,T]\rightarrow M^{n}$, which
has to be bijective for fixed $t$ and has to satisfy $\psi(\partial
M^{n},t)=\partial M^{n}$, such that the map $F$ defined by
\begin{eqnarray*}
F:M^{n}\times[0,T)\rightarrow I\times_{\lambda(r)}N^{n}:
(x,t)\mapsto \widetilde{F}(\psi(x,t),t)
\end{eqnarray*}
has the same regularity as stated in Theorem \ref{main1.1} and is
the unique solution to $(\sharp)$.
\end{lemma}

\begin{remark}
\rm{ As pointed out in \cite{Ma2}, for immersed hypersurfaces in a
Riemannian manifold and for arbitrary smooth supporting
hypersurfaces $\Sigma^n$, one can also get the short time existence
of $(\sharp)$. Naturally, for immersed hypersurfaces in a warped
product $I\times_{\lambda(r)}N^{n}$, we definitely have the short
time existence result.}
\end{remark}

Let $T^{\ast}$ be the maximal time such that there exists some $u\in
C^{2,1}(M^{n},[0,T^{\ast}))\cap C^{\infty}(M^{n},(0,T^{\ast}))$
which solves $(\widetilde{\sharp})$. In the sequel, we will prove a
priori estimates for those admissible solutions on $[0,T]$ where
$T<T^{\ast}$.

\section{$C^0$ estimate}
In this section, we will use the evolution equation of $\varphi$ to
get some estimates for $\lambda$ and $\dot \varphi$.
\begin{lemma}\label{lemma3.1}
If $\varphi$ satisfies $(\widetilde{\sharp})$, and the warping
function $\lambda(r)$ satisfies $\lambda(r)>0$, $\lambda'(r)>0$,
$\lambda''(r)\geqslant0$ on $I^{\circ}$, then we have
\begin{equation*}
\lambda(\inf
u(\cdot,0))\leqslant\lambda(u(x,t))e^{-\frac{t}{n}}\leqslant\lambda(\sup
u(\cdot,0)), \qquad\quad \forall~ x=(\xi^{1},\ldots,\xi^{n})\in
M^{n}\subset N^{n}, t\in[0,T].
\end{equation*}
\end{lemma}
\begin{proof}
If $\varphi$ could reach its maximum at the boundary, by Hopf's
Lemma, it follows that the derivative of $\varphi$ along the outward
unit normal vector must be strictly greater than $0$, which is
contradict with the boundary condition $\nabla_\mu \varphi=0$.
Therefore, $\varphi$ must attain its maximum at interior points. The
same situation happens when $\varphi$ reaches its minimum. By the
chain rule, it is easy to get
$\dot{\lambda}=\lambda'\dot{u}=\lambda'\lambda\dot{\varphi}$. Hence,
the evolution equation of $\lambda$ is the following
\begin{eqnarray} \label{3.1}
\frac{1}{\lambda'\lambda}\dot{\lambda}=\frac{v^2}{n\lambda'-\widetilde{\sigma}^{ij}\varphi_{i,j}}.
\end{eqnarray}
By the facts
\begin{eqnarray*}
\lambda_i=\lambda'u_{i}=\lambda'\lambda\varphi_i
\end{eqnarray*}
and
\begin{eqnarray*}
\lambda_{i,j}=\lambda''\lambda\varphi_{i}+(\lambda')^2\varphi_{i}\varphi_{j}+\lambda'\lambda\varphi_{i,j},
\end{eqnarray*}
we know that when $\varphi$ gets its maximum or minimum, the same
situation happens to $\lambda$. When $\lambda$ gets its maximum, the
Hessian of $\varphi$ is negative definite, which means
$\widetilde{\sigma}^{ij}\varphi_{i,j}\leqslant0$. Therefore, when
$\lambda$ gets its maximum, we have
\begin{eqnarray*}
\frac{1}{\lambda'(\sup
u(\cdot,0))\lambda}\dot{\lambda}\leqslant\frac{1}{n\lambda'(\sup
u(\cdot,0))},
\end{eqnarray*}
which implies
\begin{eqnarray*}
\frac{1}{\lambda}\dot{\lambda}\leqslant\frac{1}{n}.
\end{eqnarray*}
Integrating both sides of the above inequality, we can get
\begin{equation}\label{3.2}
\lambda\leqslant\lambda(\sup u(\cdot,0))e^\frac{t}{n}.
\end{equation}
When $\lambda$ gets its minimum,
$\widetilde{\sigma}^{ij}\varphi_{i,j}$ is positive definite. By a
similar way, we can obtain
\begin{equation}\label{3.3}
\lambda(\inf u(\cdot,0))e^\frac{t}{n}\leqslant\lambda.
\end{equation}
Combining (\ref{3.2}) and (\ref{3.3}) yields the conclusion of Lemma
\ref{lemma3.1} directly.
\end{proof}

\begin{remark}
\rm{ Clearly, Lemma \ref{lemma3.1} tells us that in the evolving
process of the IMCF $(\sharp)$, the rescaled warping function
$\lambda(u(x,t))e^{-\frac{t}{n}}$ can be controlled from both below
and above, which implies that one might expect some \emph{good}
convergence for evolving hypersurfaces after rescaling.
 }
\end{remark}

\begin{lemma}\label{lemma3.2}
If $\varphi$ satisfies $(\widetilde{\sharp})$ and $\lambda$
satisfies $\lambda(r)>0$, $\lambda'(r)>0$,
$0\leqslant\lambda^{1+\alpha}(r)\lambda''(r)\leqslant C$ for some
positive constants $\alpha$, $C$ on $I^{\circ}$, then there exist
two positive constants $C_1$ and $C_2$ such that
\begin{eqnarray*}
C_1\leqslant\dot{\varphi}\leqslant C_{2}.
\end{eqnarray*}
\end{lemma}

\begin{proof}
Set
$Q(\nabla\varphi,\nabla^2\varphi,\lambda'):=\frac{v^2}{n\lambda'-\widetilde{\sigma}^{ij}\varphi_{i,j}}$.
Differentiating both sides of the first evolution equation of
$(\widetilde{\sharp})$, it is easy to get that $\dot{\varphi}$
satisfies
\begin{eqnarray} \label{3.4}
\left\{
\begin{array}{lll}
\frac{\partial \dot{\varphi}}{\partial
t}=Q^{ij}\nabla_{ij}\dot{\varphi}+Q^{k}\nabla_k\dot{\varphi}-\frac{nv^2\lambda''\lambda\dot{\varphi}}{(n\lambda
- \widetilde{\sigma}^{ij}\varphi_{i,j})^2} \qquad \qquad & in ~M^{n}\times(0,T)\\
\nabla_ \mu \dot{\varphi}=0 \qquad\qquad & on ~\partial M^{n}\times(0,T)\\
\dot{\varphi}(\cdot,0)=\dot{\varphi}_0\qquad\qquad & on ~ M^{n}.
\end{array}
\right.
\end{eqnarray}
Similar to the argument in the proof of Lemma \ref{lemma3.1}, it
follows that $\dot{\varphi}$ must reach its maximum and the minimum
at interior points by applying Hopf's Lemma to (\ref{3.4}).
Therefore, at the point where $\dot{\varphi}$ gets its maximum, we
have
\begin{eqnarray}  \label{3.5}
\dot{\varphi}_i=0,   \qquad \dot{\varphi}_{i,j}\leqslant0.
\end{eqnarray}
Conversely, at the point where $\dot{\varphi}$ gets its minimum, we
have
\begin{eqnarray} \label{3.6}
\dot{\varphi}_i=0,   \qquad \dot{\varphi}_{i,j}\geqslant0.
\end{eqnarray}
Set
\begin{eqnarray*}
\dot{\varphi}_{max}:=\sup\limits_{M^n}\dot{\varphi}(\cdot,t), \qquad
\dot{\varphi}_{min}:=\inf\limits_{M^n}\dot{\varphi}(\cdot,t).
\end{eqnarray*}
Combining (\ref{3.4}), (\ref{3.5}) and the assumptions for the
warping function $\lambda$, we can obtain the evolution equation of
$\dot{\varphi}_{max}$ as follows
\begin{eqnarray*}
\frac{\partial \dot{\varphi}_{max}}{\partial
t}=Q^{ij}\nabla_{ij}\dot{\varphi}_{max}-\frac{n\lambda''\dot{\varphi}_{max}}{\lambda
H^2} \leqslant -\frac{n\lambda''\lambda
\dot{\varphi}^3}{v^2}\leqslant 0,
\end{eqnarray*}
which implies
\begin{eqnarray} \label{3.7}
\dot{\varphi}\leqslant \dot{\varphi}_{max}\leqslant
\sup\limits_{M^n}\dot{\varphi}(\cdot,0):=C_{2}.
\end{eqnarray}
On the other hand, since
$0<\lambda^{1+\alpha}(r)\lambda''(r)\leqslant C$, we have
\begin{eqnarray*}
0<\lambda''(r)\lambda(r)\leqslant Ce^{-\frac{\alpha}{n}t}
\end{eqnarray*}
by applying Lemma \ref{lemma3.1}. \emph{In general, the constant $C$
in the above inequality should be different from the one in the
assumption of Lemma \ref{lemma3.2}. However, for convenience, we use
the same symbol}. Then, combining (\ref{3.4}), (\ref{3.6}) and the
assumptions for $\lambda$, we have
\begin{eqnarray*}
\frac{\partial \dot{\varphi}_{min}}{\partial
t}=Q^{ij}\nabla_{ij}\dot{\varphi}_{min}-\frac{n\lambda''\lambda
\dot{\varphi}^3_{min}}{v^2}\geqslant -\frac{n\lambda''\lambda
\dot{\varphi}^3_{min}}{v^2}\geqslant
-Ce^{-\frac{\alpha}{n}t}\dot{\varphi}
\end{eqnarray*}
by applying Lemma \ref{lemma3.1}. The maximum principle tells us
that $\dot{\varphi}$ is bounded from below by the solution of
 ordinary differential equation (we write as ODE for short)
\begin{equation} \label{3.8}
\frac{\partial f}{\partial t}=-Ce^{-\frac{\alpha}{n}t}f,
\end{equation}
with $f(0):=\inf\limits_{M^{n}\subset N^n} \dot{\varphi}(\cdot,0)$.
By a straight calculation, we get the solution of the ODE
(\ref{3.8}) as follows
\begin{eqnarray*}
f=f(0)e^{\frac{C\alpha}{n}(e^{-\frac{\alpha}{n}t}-1)}\geqslant
C_{1}:=f(0)e^{-\frac{C\alpha}{n}}>0.
\end{eqnarray*}
Therefore, we have
\begin{eqnarray*}
\dot{\varphi}\geqslant \dot{\varphi}_{min}(t) \geqslant f(t)
\geqslant C_1.
\end{eqnarray*}
Together with (\ref{3.7}), the conclusion of Lemma \ref{lemma3.2}
follows.
\end{proof}

\begin{remark}
\rm{By the $C^0$-estimate, we know that IMCF preserves the convexity
during the evolving process, which implies $H>0$ for all
$t\in[0,T]$. This fact has been shown in \cite[Theorem 3.5]{Zh}.}
\end{remark}

\section{Gradient Estimate}

In this section, the gradient estimate will be shown.

\begin{lemma}\label{lemma4.1}
If $\varphi$ satisfies $(\widetilde{\sharp})$, $\lambda$ satisfies
$\lambda(r)>0$, $\lambda'(r)>0$,
$0\leqslant\lambda^{1+\alpha}(r)\lambda''(r)\leqslant C$ for some
positive constants $\alpha$, $C$ on $I^{\circ}$, and the Ricci
curvature of $N^n$ is nonnegative, then we have
\begin{eqnarray*}
|\nabla\varphi|\leqslant C_3,
\end{eqnarray*}
where $C_3$ is a nonnegative constant depending on
$\sup\limits_{M^{n}}\varphi(\cdot,0)$, i.e., the supremum of
$\varphi(x,t)$ at the initial time $t=0$.
\end{lemma}

\begin{proof}
Set $\psi=\frac{|\nabla\varphi|^2}{2}$. Then differentiating $\psi$
with respect to $t$ and together with $(\widetilde{\sharp})$, we
have
\begin{eqnarray*}
\frac{\partial \psi}{\partial t}=\frac{\partial}{\partial
t}\nabla_m\varphi\nabla^m\varphi=\nabla_m
\dot{\varphi}\nabla^m\varphi=\nabla_m Q\nabla^m\varphi,
\end{eqnarray*}
which is equivalent with
\begin{eqnarray*}
\frac{\partial \psi}{\partial
t}=\left(Q^{ij}\nabla_{ijm}\varphi+Q^k\nabla_{mk}\varphi-\frac{v^2\lambda^{\prime
\prime}\lambda\nabla_m\varphi}{(n\lambda^{\prime}-\tilde{\sigma}^{ij}\varphi_{i,j})^2}\right)\nabla^m\varphi.
\end{eqnarray*}
By straightforward calculation, we have
\begin{eqnarray} \label{4.1}
\frac{\partial \psi}{\partial
t}=Q^{ij}\nabla_{ijm}\varphi\nabla^m\varphi+Q^k\nabla_{km}\nabla^m\varphi-\frac{2n\lambda^{\prime
\prime}\psi}{\lambda H^2}
\end{eqnarray}
and
\begin{eqnarray}   \label{4.2}
\nabla_{ij}\psi=\nabla_j(\nabla_{mj}\varphi\nabla^m\varphi)&=&\nabla_{mij}\varphi\nabla^m\varphi+\nabla_{mi}\varphi\nabla^m_j\varphi\nonumber\\
&=&(\nabla_{ijm}\varphi+R^l_{imj})\nabla^m\varphi+\nabla_{mi}\nabla^m_j\varphi.
\end{eqnarray}
On the other hand, applying the Ricci identity, we can obtain
\begin{eqnarray} \label{4.3}
\nabla_{ijm}\varphi\nabla^m\varphi=\nabla_{ij}\psi-R^l_{imj}\nabla_l\varphi\nabla^m\varphi-\nabla_{mi}\varphi\nabla^m_j\varphi.
\end{eqnarray}
Combining (\ref{4.1}), (\ref{4.2}) and (\ref{4.3}) yields
\begin{eqnarray} \label{4.4}
\frac{\partial \psi}{\partial
t}=Q^{ij}\nabla_{ij}\psi+\left(Q^k-\frac{\nabla^k\psi}{\lambda^2v^2H^2}\right)\nabla_k
\psi-\frac{1}{\lambda^2H^2}\sigma^{ij}R_{limj}\nabla^l\varphi\nabla^m\varphi+  \nonumber\\
\frac{R_{limj}\nabla^i\varphi\nabla^j\varphi\nabla^l\varphi\nabla^m\varphi}{\lambda^2H^2v^2}-\frac{|\nabla^2\varphi|^2}{\lambda^2H^2}-\frac{2n\lambda''\psi}{\lambda
H^2}.  \qquad\qquad
\end{eqnarray}
Since the curvature tensor is antisymmetric with respect to the
indices $i$ and $j$, we have
$R_{limj}\nabla^l\varphi\nabla^i\varphi\nabla^m\varphi\nabla^j\varphi=0$.
Besides, the nonnegativity of the Ricci curvature yields
\begin{eqnarray*}
\sigma^{ij}R_{limj}\nabla^l\varphi\nabla^m\varphi=Ric(\nabla\varphi,\nabla\varphi)\geqslant0.
\end{eqnarray*}
By Lemmas \ref{lemma3.1}, \ref{lemma3.2} and the fact
$\dot\varphi=\frac{v}{\lambda H}$, we have
 \begin{eqnarray*}
 \lambda H^2=\lambda\left(\frac{v}{\lambda\dot\varphi}\right)^2&=&\frac{\lambda^2+|\nabla
 r|^2}{\lambda\dot\varphi^2}\\
 &\leqslant&\frac{\lambda^2+1}{\lambda\dot\varphi^2}\\
 &\leqslant&\frac{\lambda^2+1}{\lambda C_{1}^2},
 \end{eqnarray*}
 which implies
\begin{eqnarray*}
-\frac{2n\lambda''}{\lambda H^2}&\geqslant&-
\frac{2nC_{1}^2\lambda\lambda''}{\lambda^2+1}\geqslant
-2nC_{1}^2\lambda''\lambda^{-1}\geqslant-2nC_{1}^2C\lambda^{-(2+\alpha)}\\
&\geqslant&-2nC_{1}^2C \left(\lambda\left(\inf_{M^n}
u(\cdot,0)\right)\right)^{-(2+\alpha)}e^{-\frac{(2+\alpha)t}{n}}\\
 &=&-C_{4}e^{-\frac{(2+\alpha)t}{n}},
\end{eqnarray*}
where $C_4:=2nC_{1}^2C \left(\lambda\left(\inf\limits_{M^n}
u(\cdot,0)\right)\right)^{-(2+\alpha)}$. Putting the above three
facts into (\ref{4.4}) results in
\begin{eqnarray} \label{4.5}
\frac{\partial \psi}{\partial t}\leqslant
Q^{ij}\nabla_{ij}\psi+\left(Q^k-\frac{\nabla^k\psi}{\lambda^2v^2H^2}\right)\nabla_k
\psi-C_4e^{-\frac{(2+\alpha)t}{n}}\psi.
\end{eqnarray}
Choose an orthonormal frame $\{e_{1},\cdots,e_{n}\}$ at
$x\in\partial M^n$ such that $e_{1},\cdots,e_{n-1}\in T_{x}\partial
M^n$ and $e_{n}=\mu$, and then we can obtain
\begin{eqnarray*}
\nabla_ \mu
\psi&=&\nabla_{e_n}\psi=\sum_{i=1}^{n}\nabla_{e_{i},e_{n}}\varphi\nabla_{e_{i}}\varphi=\sum_{i=1}^{n-1}(\nabla_{e_i}\nabla_{e_n}\varphi-(\nabla_{e_i}e_n)\varphi)\nabla_{e_i}\varphi\\
&=&-\sum_{i=1}^{n-1}\langle\nabla_{e_i}e_n,e_j\rangle\nabla_{e_j}\varphi\nabla_{e_i}\varphi\\
&=&-\sum_{i=1}^{n-1}h_{ij}^{\partial{M}^n}\nabla_{e_i}\varphi\nabla_{e_j}\varphi\leqslant0,
\end{eqnarray*}
where $h_{ij}^{\partial{M}^n}$ is the second fundamental form of
$\partial M^n$. The last inequality holds because of the convexity
of the cone $\Sigma^n$. Therefore, together with (\ref{4.5}),we know
that $\psi$ satisfies
\begin{eqnarray*}
\left\{
\begin{array}{lll}
\frac{\partial \psi}{\partial t}\leqslant
Q^{ij}\nabla_{ij}\psi+\left(Q^k-\frac{\nabla^k\psi}{\lambda^2v^2H^2}\right)\nabla_k
\psi-C_4e^{-\frac{(2+\alpha)t}{n}}\psi
\qquad \qquad & in ~M^{n}\times(0,T)\\
\nabla_\mu \varphi \leqslant 0 \qquad\qquad & on ~\partial M^{n}\times(0,T)\\
\psi(\cdot,0)=\frac{|\nabla\varphi(\cdot,0)|^2}{2}=\frac{|\nabla\varphi_{0}|^2}{2}
\qquad\qquad & on ~ M^{n}.
\end{array}
\right.
\end{eqnarray*}
Then using the maximum principle, we have
$\psi\leqslant\sup\limits_{M^n}\frac{|\nabla\varphi_{0}|^2}{2}$,
which, together with Lemma \ref{lemma3.1}, yields the desired
gradient estimate.
\end{proof}

\section{Higher regularity and Convergence}

In this section, the convergence and the higher regularity of the
IMCF $(\sharp)$ will be discussed after the area-preserving
rescaling. We consider the rescaling $ \widehat F=\eta (t)F(x,t)$,
where $F$ is the parameterization of the graph $M^{n}_{t}$, and
$\eta(t)$ is the smooth function with respect to $t$ satisfying
\begin{eqnarray} \label{5.1}
\int_{\widehat{M}^{n}_{t}}d\widehat\mu =|M_0|,
\end{eqnarray}
where $\widehat{M}^{n}_{t}$ is the rescaled hypersurface,
$d\widehat{\mu}$ is the volume element of $\widehat{M}^{n}_{t}$,
$|M_0|$ denotes the area of the initial hypersurface $M_0$. Recall
that the induced metric and the second fundamental form of
$M^{n}_{t}$ are given by
\begin{eqnarray*}
g_{ij}=\left\langle\frac{\partial F}{\partial x^{i}},\frac{\partial
F}{\partial x^{j}}\right\rangle, \qquad\qquad
h_{ij}=\left\langle\vec{\nu},\frac{\partial^2 F}{\partial x^i
\partial x^j}\right\rangle,
\end{eqnarray*}
so the corresponding induced metric and the second fundamental form
of $\widehat{M}^{n}_{t}$ should be
\begin{eqnarray*}
\widehat{g}_{ij}=\eta^{2}g_{ij}, \qquad
\widehat{h}_{ij}=\eta^{2}h_{ij}, \qquad \widehat{g}^{ij}=\eta^{-2}
g^{ij}.
\end{eqnarray*}
Differentiating both sides of (\ref{5.1}) and then we can get
\begin{eqnarray*}
\frac{d}{dt}\int _{\widehat{M}^{n}_{t}}d\widehat\mu=\int
_{M^n_t}\frac{\widehat{g}^{ij}\frac{\partial}{\partial
t}\widehat{g}_{ij}}{2}d\mu=\int
_{M^n_t}\left(\eta\eta'g_{ij}+\frac{1}{H}h_{ij}\eta^2\right)g^{ij}\eta^{-2}d\mu=0,
\end{eqnarray*}
which implies that
\begin{eqnarray*}
\int _{M^n_t} \left(n\eta^{-1}\eta' +1 \right)d\mu =0.
 \end{eqnarray*}
Therefore, we have
\begin{eqnarray*}
n\eta^{-1}\eta'+1=0,
\end{eqnarray*}
and then solving the above ODE, together with $\eta(0)=1$, yields
$\eta(t)=e^{-\frac{1}{n}t}$.

In order to get the long time existence for the IMCF $(\sharp)$, we
do the rescaling $\widehat{u}=ue^{-\frac{1}{n}t}$ for the graphic
function $u(x,t)$ of the evolving hypersurface $M^n_t$ in the base
part of the warped product $I\times_{\lambda(r)}N^{n}$. Through this
process, we can obtain the following result.

\begin{lemma} \label{lemma5.1}
Let $u$ be an admissible solution of $(\widetilde{\sharp})$ and let
$\Sigma ^{n}$ be a smooth, convex cone. If $\lambda$ satisfies
$\lambda(r)>0$, $0<\lambda'(r)\leqslant C$,
$0\leqslant\lambda^{1+\alpha}(r)\lambda''(r)\leqslant C$ for some
positive constants $\alpha$, $C$ on $I^{\circ}$, and the Ricci
curvature of $N^n$ is nonnegative, then there exist some $\beta
>0 $ and some $D>0$ such that
\begin{eqnarray*}
[\nabla \widehat{u}]_{\beta}+[\partial{\widehat u}/\partial
t]_{\beta}+[\widehat H]_{\beta} \leqslant D \left(\parallel
u_0\parallel_{C^{2+\alpha}(M^{n})}, n, \beta, M^{n}\right)
\end{eqnarray*}
where $[f]_{\beta}:=[f]_{x,\beta}+[f]_{t,\beta/2}$ is the sum of the
H\"older coefficients of $f$ with respect to $x$ and $t$ in the
domain $M^{n} \times [0,T]$.
\end{lemma}

\begin{proof}
First, we try to have the priori estimates for $|\nabla \varphi|$
and $|\partial \varphi/\partial t|$. That is because the priori
estimates for $|\nabla \varphi|$ and $|\partial \varphi/\partial t|$
imply a bound for $[\widehat u]_{x,\beta}$ and $[\widehat
u]_{t,\beta/2}$, which, together with \cite[Chapter 2, Lemma
3.1]{La1}, can give the bound for $[\nabla \widehat u]_{t,\beta/2}$
provided a bound for $[\nabla\widehat{u}]_{x,\beta}$ obtained.
Since, after rescaling, $|\nabla\widehat u|$ and $\partial \widehat
u /\partial t$ can be written as
\begin{eqnarray*}
\nabla \widehat u=\nabla u \cdot e^{-t/n}, \qquad \dot
{\widehat u}=\dot
ue^{-\frac{1}{n}t}-\frac{1}{n}\widehat{u},
\end{eqnarray*}
which implies $\nabla \widehat u=\lambda \nabla \varphi \cdot
e^{-\frac{\alpha}{n}t}$. Then it is sufficient to bound $[\nabla
\varphi]_{x,\beta}$ if one wants to bound
$[\nabla\widehat{u}]_{x,\beta}$. In order to get this bound, we fix
$x$ and rewrite the first evolution equation of
$(\widetilde{\sharp})$ as follows
\begin{eqnarray}  \label{5.2}
\mathrm{div}_{\sigma}\left(\frac{\nabla \varphi}{\sqrt{1+|\nabla
\varphi|^2}}\right)=\frac{n\lambda'}{\sqrt{1+|\nabla
\varphi|^2}}-\frac{\sqrt{1+|\nabla \varphi|^2}}{\dot \varphi},
\end{eqnarray}
which is an elliptic PDE endowed with the Neumann boundary condition
$\nabla_\mu \varphi =0$. Clearly, the RHS of (\ref{5.2}) is bounded
since $\dot\varphi$, $\nabla\varphi$ are bounded (see Lemmas
\ref{lemma3.2} and \ref{lemma4.1} respectively) and
$0<\lambda'\leqslant C$. Besides, the RHS of (\ref{5.2}) is also a
measurable function in $x$. Therefore, by similar calculations to
those in \cite{La2}, Chapter 4, $\S6$ (interior estimate) and
Chapter 10, $\S2$ (boundary estimate), we can get a Morrey estimate,
which yields the estimate for $[\nabla \varphi]_{x,\beta}$.

Note that $\dot {\widehat u}=\lambda \dot \varphi e^{-\alpha t/n}$.
So, we need to bound $\dot \varphi$ if we want to get a bound for
$[\partial \widehat u /\partial t]_{\beta}$. By the first equation
of (\ref{3.4}), it is not difficult to get the evolution equation
for $\dot\varphi$ with respect to the induced rescaled metric as
follows
\begin{eqnarray*}
\frac{\partial{\dot \varphi}}{\partial t}=\mathrm{div}_{\widehat
g}\left(\frac{\nabla \dot \varphi}{{\widehat H}^2}\right)-2\dot
\varphi^{-1}\frac{|\nabla \dot \varphi|^2_{\widehat g}}{{\widehat
H}^2}-n\frac{\lambda\lambda''{\dot \varphi}^3}{v^2}.
\end{eqnarray*}
Note that the Neumann condition $\nabla _{\mu}\dot \varphi =0$
implies that the interior and boundary estimates are basically the
same. Then, together with the strict positivity of $\dot \varphi$,
we can define the test function $\chi=\xi^2\dot\varphi$ and the
support of $\xi$ has to be chosen away from the boundary for the
interior estimate. Integration by parts and Young's inequality
results in
\begin{eqnarray*}
&&\frac{1}{2}\parallel \dot \varphi \xi
\parallel^{2}_{2,M_{t}^n}\Big{|}_{t_0}^{t_1}+\frac{1}{\max{\widehat H}^2}\int ^{t_1}_{t_0}\int _{M^n_t}{\xi}^2|\nabla \dot \varphi|^2d\mu_t
dt\leqslant \\
 && \qquad \qquad \qquad \int ^{t_1}_{t_0}\int _{M^n_t}\left[\dot
{\varphi}^2\xi\left|\frac{\partial \xi}{\partial
t}\right|+\frac{{\xi}^2|\nabla \dot \varphi|^2}{2\max{\widehat
H}^2}+\frac{2\max{\widehat H}^2\dot{\varphi}^2|\nabla
\xi|^2}{\min{\widehat H}^4}\right],
\end{eqnarray*}
which implies
\begin{eqnarray*}
&&\frac{1}{2}\parallel \dot \varphi \xi
\parallel^{2}_{2,M_{t}^n}\Big{|}_{t_0}^{t_1}+\frac{1}{2\max{\widehat H}^2}\int ^{t}_{t_0}\int _{M^n_t}{\xi}^2|\nabla \dot \varphi|^2d\mu_t
dt\leqslant \\
 && \qquad \qquad \qquad \left(1+\frac{2\max{\widehat H}^2}{\min{\widehat
H}^4}\right)\int ^{t_1}_{t_0}\int _{M^n_t}\dot
{\varphi}^2\left[\xi\left|\frac{\partial \xi}{\partial
t}\right|+|\nabla \xi|^2\right],
\end{eqnarray*}
 where, as before, $\max(\cdot)$ and $\min(\cdot)$
denote the supremum and the infimum of a prescribed quantity over
$M^{n}\subset N^n$ respectively. Then, similar to \cite{La1},
Chapter 5, $\S1$ (interior estimate) and $\S7$ (boundary estimate),
the boundedness for $[\dot \varphi]_{\beta}$ can be obtained, and
moreover, all local interior and boundary estimates are independent
of $T$.

The estimate of $\widehat H$ follows from the estimates for
$\lambda$, $\nabla \varphi$, $\dot \varphi$ and the identity
$\dot\varphi\lambda e^{-t/n}\widehat{H}=v=\sqrt{1+|\nabla
\varphi|^2}$.
\end{proof}

Applying Lemma \ref{lemma5.1}, we can get the following higher-order
estimates.
\begin{lemma} \label{lemma5.2}
Let $u$ be an admissible solution of $(\widetilde{\sharp})$ and let
$\Sigma ^{n}$ be a smooth, convex cone. If $\lambda$ satisfies
$\lambda(r)>0$, $0<\lambda'(r)\leqslant C$,
$0\leqslant\lambda^{1+\alpha}(r)\lambda''(r)\leqslant C$ for some
positive constants $\alpha$, $C$ on $I^{\circ}$, and the Ricci
curvature of $N^n$ is nonnegative, then for every $t_0 \in (0,T)$,
there exists some $\beta>0$ such that
\begin{eqnarray*}
\parallel \widehat u \parallel _{C^{2+\beta,1+\beta/2}(M^{n}\times[0,T])}\leqslant D\left(\parallel u_0 \parallel_{C^{2+\alpha}(M^n)},n,\beta,M^n\right)
\end{eqnarray*}
and
\begin{eqnarray*}
\parallel \widehat u \parallel _{C^{2k+\beta,k+\beta/2}(M^{n}\times[t_0,T])}\leqslant D\left(\parallel
u(\cdot,t_0)
\parallel_{C^{2+\alpha}(M^n)},n,\beta,M^n\right).
\end{eqnarray*}
\end{lemma}
\begin{proof}
Since $\varphi(x,t)=\int_{c}^{u(x,t)}\frac{1}{\lambda(s)}ds$, we
know that the bound of $\varphi$ leads to the bound of $u$. So, we
try to estimate $\varphi$. Rewrite $(\widetilde\sharp)$ as follows
\begin{eqnarray} \label{5.3}
\frac{\partial \varphi}{\partial t}=\frac{1}{\widehat H^2}\Delta
_{\widehat g}\varphi +\left(\frac{2\sqrt{1+|\nabla
\varphi|^2}}{\widehat \lambda \widehat H}-\frac{n
\lambda'}{{\widehat \lambda}^2 {\widehat H}^2}\right)
\end{eqnarray}
By Lemma \ref{lemma5.1}, we know that (\ref{5.3}) is a uniformly
parabolic PDE with H\"{o}lder continuous coefficients. Therefore, by
\cite[Chapter IV, Theorem 5.3]{La1}, which shows $\varphi$ is
$C^{2+\beta,1+\beta/2}$, and the linear theory in \cite[Chapter
4]{Li}, which implies the second-order bound, we have
\begin{eqnarray} \label{5.4}
\parallel \widehat u \parallel _{C^{2+\beta,1+\beta/2}}\leqslant D\left(\parallel u_0
\parallel_{C^{2+\alpha}(M^n)},n,\beta,M^n\right).
\end{eqnarray}
Differentiating both sides of (\ref{5.3}) with respect to $t$ and
$\xi_{i}$, $1\leqslant i \leqslant n$, respectively, one can easily
get evolution equations of $\dot\varphi$ and $\varphi_i$
respectively, which, using the estimate (\ref{5.4}), can be treated
as uniformly parabolic PDEs on the time interval $[t_0,T]$. At the
initial time $t_0$, all compatibility conditions are satisfied and
the initial function $\varphi(\cdot,t_0)$ is smooth, which implies a
$C^{3+\beta,(3+\beta)/2}$ estimate for $\varphi_i$ and a
$C^{1+\beta,1+\beta/2}$ estimate for $\dot\varphi$. So, we have the
$C^{4+\beta,2+\beta/2}$ estimate for $\widetilde{u}$. From
\cite[Chapter 4, Theorem 4.3, Exercise 4.5]{Li} and the above
argument, it is not difficult to know that the constant are
independent of $T$. Higher regularity can be proven by induction
over $k$.
\end{proof}

Then the long-time existence and the convergence can be discussed.

\begin{lemma} \label{lemma5.3}
let $\Sigma ^{n}$ be a smooth, convex cone. Let $u$ be an admissible
solution of $(\widetilde{\sharp})$ and let $T^\ast$ be the maximal
existence time. Then $T^{\ast}=\infty$, and the rescaled solution
$\widetilde{F}(\cdot,t)=F(\cdot,t)e^{-\frac{1}{n} t}$
($F(\cdot,t)$ is the embedding map mentioned in Theorem
\ref{main1.1}) converges smoothly to an embedding $F_{\infty}$,
mapping $M^n$ into a piece of a geodesic sphere.
\end{lemma}
\begin{proof}
By Lemma \ref{lemma5.2}, we know that the H\"{o}lder norms of
$u=\widehat{u}e^{t/n}$ cannot blow up as $T$ tends to the
maximal time $T^{\ast}<\infty$, which implies that $u$ can be
extended to a solution to $(\widetilde\sharp)$ in $[0,T^{\ast}]$.
The short time existence result (see Lemma \ref{lemma2.1}) and the
higher-order estimates (see Lemma \ref{lemma5.2}) imply the
existence of a solution beyond $[0,T^{\ast}]$ which is smooth away
from $t=0$. This is a contradiction. Therefore, we have
$T^{\ast}=\infty$.

By Lemmas \ref{lemma3.1} and \ref{lemma4.1}, we can get the
following estimate
\begin{eqnarray*}
|Du|\leqslant Ke^{-\gamma t},
\end{eqnarray*}
where $K$ is a positive constant depending only on $C_3$
$\lambda(\inf u(\cdot,0))$, $\gamma$ is a constant depending only on
$\alpha$ and the dimension $n$. By the Arzel\`a-Ascoli theorem,we
know that every subsequence of $\widehat u$ converges to a constant
function $\omega_{\infty}$ in $C^{1}(M^n)$. Assume that $\widehat u$
convergent to the constant $\omega_{\infty}$ in $C^{k}(M^n)$. Since
$\widehat u$ is uniformly bounded in $C^{k+\beta +1}$ and $\widehat
u$ is H\"older continuous, by the Arzel\`a-Ascoli theorem we know
that there exists a subsequence convergent to  $\omega_{\infty}$ in
$C^{k+1}(M^n)$. Then we can get the conclusion that every
subsequence must converge, and the limit has to be
$\omega_{\infty}$. Therefore, $\widehat u$ converges  to
$\omega_{\infty}$ in $C^{k+1}(M^n)$. The $C^{\infty}$ convergence
follows by the induction.

Finally, we accurately describe the asymptotic behavior of rescaled
hypersurfaces as time tends to infinity. Recall that
$h^i_j=\frac{1}{\lambda v}(\lambda'\delta ^i_j-\tilde
\sigma^{ik}\varphi _{kj})$. So, by the assumption
$0<\lambda'\leqslant C$, Lemmas \ref{lemma3.1} and \ref{lemma5.2},
it follows that
\begin{eqnarray*}
\left|\widehat h ^i_j-\delta^i_j\right|&\leqslant&
e^{-\frac{1}{n}t}\left(\frac{\lambda'}{\lambda
v}-1\right)\delta^i_j+\frac{1}{\lambda
v}e^{\frac{1}{n}t}\widetilde \sigma ^{ik}\varphi
_{kj}\\
&\leqslant& e^{-(\frac{n+1}{n})t}\left(\frac{\lambda'}{v\lambda
e^{-\frac{1}{n}t}}\delta^i_j+\frac{1}{v\lambda
e^{-\frac{1}{n}t}}\widetilde \sigma ^{ik}\varphi
_{kj}\right)\\
&\leqslant& C_{5}e^{-\delta t}
\end{eqnarray*}
for some positive constant $C_{5}$ depending on $C$, $D$, $C_3$,
$\lambda(\inf u(\cdot,0)) $, and some constant $\delta$ depending on
$\alpha$ and $n$. So, from the above argument, we know that after
rescaling, the evolving hypersurfaces converge smoothly to a piece
of a geodesic sphere as time tends to infinity.
\end{proof}

Theorem \ref{main1.1} follows naturally from Lemmas \ref{lemma2.1}
and \ref{lemma5.3}.

\vspace{0.5 cm}

$\\$\textbf{Acknowledgments}. This research was supported in part by
the National Natural Science Foundation of China (Grant Nos.
11201131, 11401131 and 11101132) and Hubei Key Laboratory of Applied
Mathematics (Hubei University).

\end{document}